\documentclass[a4j,11pt]{article}
\usepackage{amsmath,amsthm,amssymb,amscd,ascmac}
\usepackage{graphicx, mathrsfs}
\theoremstyle{definition}
\numberwithin{equation}{section}
\newtheorem{thm}{Theorem}[section]

\newtheorem{lem}[thm]{Lemma}
\newtheorem{cor}[thm]{Corollary}

\allowdisplaybreaks[4]

\theoremstyle{definition}

\newtheorem{rem}[thm]{Remark}

\usepackage{bm}

\begin{document}
\title{Some singular values of the elliptic lambda function and incredible cubic identities}
\author{Genki Shibukawa\thanks{This work was supported by Grant-in-Aid for JSPS Fellows (Number 18J00233).}
}
\date{
\small MSC classes\,:\,11F03, 11G16, 11R16}
\pagestyle{plain}

\maketitle

\begin{abstract}
We propose three kinds of explicit formulas for the elliptic lambda function by the elliptic modular function. 
Further, we derive incredible cubic identities as a corollary of our explicit formulas and evaluate some singular values of the elliptic lambda function explicitly.
\end{abstract}

\section{Introduction}
Let $\tau $ be a complex number with a positive imaginary part and $\mathbb{Z}$ be the ring of integers. 
We define the elliptic lambda function $\lambda (\tau )$ and elliptic modular function $j(\tau )$ by 
\begin{align}
\lambda (\tau )
   &:=
   16q^{\frac{1}{2}}
      \prod_{n=1}^{\infty}\frac{(1+q^{n})^{8}}{(1+q^{n-\frac{1}{2}})^{8}}
   =
   16(q^{\frac{1}{2}}-8q+44q^{\frac{3}{2}}-192q^{2}+718q^{\frac{5}{2}}-\cdots), \nonumber \\
j(\tau )
   &:=
   2^{8}\frac{(1-\lambda (\tau )+\lambda (\tau )^{2})^{3}}{\lambda (\tau )^{2}(1-\lambda (\tau ))^{2}}
   =
   \frac{1}{q}+744+196884q+21493760q^{2}+\cdots \nonumber 
\end{align}
respectively. 
Here we put $q:=e^{2\pi \sqrt{-1} \tau}$. 
The elliptic modular function $j(\tau )$ is a modular function for the modular group 
\begin{align}
\Gamma 
   :=&
\mathrm{SL}_{2}(\mathbb{Z})
   =\!
   \left\{\!\!
   \begin{pmatrix}
   a & b \\
   c & d
   \end{pmatrix} \! \Bigg\vert \, a,b,c,d \in \mathbb{Z}, \,ad-bc=1\right\}\!
   =
   \left\langle \! \begin{pmatrix} 1 & 1 \\ 0 & 1\end{pmatrix}\!,\! \begin{pmatrix} 0 & -1 \\ 1 & 0\end{pmatrix} \! \right\rangle \nonumber
\end{align}
and the elliptic modular function $\lambda (\tau )$ is a modular function for the principal congruence subgroup of level $2$
\begin{align}
\Gamma (2)
   :=&
   \left\{
   \begin{pmatrix}
   a & b \\
   c & d
   \end{pmatrix} \in \mathrm{SL}_{2}(\mathbb{Z}) \,\Bigg\vert \,a\equiv d\equiv 1 \,(\mathrm{mod}\,2),\,\,b\equiv c\equiv 0 \,(\mathrm{mod}\,2)\right\} \nonumber \\
   =&
   \left\langle \begin{pmatrix} 1 & 2 \\ 0 & 1\end{pmatrix}, \,\,\begin{pmatrix} 1 & 0 \\ 2 & 1\end{pmatrix}, \,\,\begin{pmatrix} -1 & 0 \\ 0 & -1\end{pmatrix}\right\rangle. \nonumber
\end{align}
Since the generators of the modular group $\Gamma $ act by 
$$
\lambda (\tau +1)
   =
   \frac{\lambda (\tau )}{\lambda (\tau )-1}
$$
and 
$$
\lambda \left(-\frac{1}{\tau }\right)
   =
   1-\lambda (\tau ), 
$$
the elliptic lambda function $\lambda (\tau )$ on the coset $\Gamma /\Gamma (2)\simeq \mathrm{SL}_{2}(\mathbb{Z}/2\mathbb{Z})$ which is isomorphic to the symmetric group of degree $3$ has the following six values: 
\begin{align}
\lambda_{1}\left(\tau \right)
   &:=
   \lambda (\tau), \,\,
\lambda_{2}\left(\tau \right)
   :=
   \lambda \left(\frac{\tau -1}{\tau }\right)
   =
   \frac{\lambda (\tau)-1}{\lambda (\tau)}, \nonumber \\ 
\lambda_{3}\left(\tau \right)
   &:=
   \lambda \left(\frac{1}{1-\tau }\right)
   =
   \frac{1}{1-\lambda (\tau)}, \,\,
\lambda_{4}\left(\tau \right)
   :=
   \lambda \left(-\frac{1}{\tau }\right)
   =
   1-\lambda (\tau), \nonumber \\
\lambda_{5}\left(\tau \right)
   &:=
   \lambda \left(\frac{\tau }{1-\tau }\right)
   =
   \frac{1}{\lambda (\tau)}, \,\, 
\lambda_{6}\left(\tau \right)
   :=
   \lambda \left(\tau -1\right)
   =
   \frac{\lambda (\tau)}{\lambda (\tau)-1}. \nonumber 
\end{align}
Thus symmetric functions in variables $\lambda_{1}\left(\tau \right), \ldots, \lambda_{6}\left(\tau \right)$ are $\Gamma $-invariant.

These functions $\lambda (\tau)$ and $j(\tau )$ are classical examples of modular functions and have been investigated since the 19th century. 
One of the topics is the study of singular values (or singular moduli) which are the special values of these modular functions at imaginary quadratic arguments $\tau$ in the upper half complex plane. 
Even in this topic, there are many previous studies by Hermite, Kronecker, Russell, Smith, Berwick, Weber, Ramanujan, Watson, Berndt, and many others. See \cite{B1}, \cite{C}, \cite{S} and \cite{W}. 
In particular, H. Weber \cite{W} evaluated some $j(\tau)$ whose singular values are integer, and wrote these singular values explicitly. 
Let $j_{d}$ denote the values of $j(\tau )$ at $\tau=\frac{1+\sqrt{-d}}{2}$ for $d>0$. 
We write down all $j_{d}$ in the list of Weber \cite{W} p462--462, Cox \cite{C} p261 or Silverman \cite{S} p483 
\begin{align}
& 
j_{3}=0, \,\,
j_{7}=-15^{3}, \,\,
j_{11}=-32^{3}, \,\, 
j_{19}=-96^{3}, \,\,
j_{27}=-3\cdot 160^{3}, \nonumber \\
\label{eq:Weber}
& 
j_{43}=-960^{3}, \,\,
j_{67}=-5280^{3}, \,\,
j_{163}=-640320^{3}. 
\end{align}
This table is very familiar. 
In fact, it is kell-known that these are all singular values of $j_{d}$ whose values are integers. 

Similarly, 
R. Russell \cite{R} and W. Berwick \cite{B} gave singular values of $j(\tau)$ in terms of quadratic (and cubic) irrationalities. 
Under the following, along with Berwick's original results \cite{B} p57--59 (the first equality of each), simplified expressions of them (the second equality of each) are also given. 
\begin{align}
j_{4}
   &=
   \left(-3\sqrt{2}(1-\sqrt{2})^{4}(3-4\sqrt{2})(-5-6\sqrt{2})\right)^{3} 
   =
   3^{3}(724-513\sqrt{2})^{3}, \nonumber \\
j_{5}
   &=
   \left(-2^{2}\sqrt{5}\left(\frac{1-\sqrt{5}}{2}\right)^{3}\left(3\frac{1-\sqrt{5}}{2}-2\right)\right)^{3} 
   =
   2^{3}(25-13\sqrt{5})^{3}, \nonumber \\
j_{9}
   &=
   (2+\sqrt{3})\left(-2^{2}\sqrt{3}(2-\sqrt{3})(1-2\sqrt{3})(2-3\sqrt{3})\right)^{3} \nonumber \\
   &=
   4^{3}(2+\sqrt{3})(102-61\sqrt{3})^{3}, \nonumber \\
j_{13}
   &=
   \left(2^{2}\cdot 3\cdot 5\left(\frac{3-\sqrt{13}}{2}\right)^{2}\left(3\frac{1-\sqrt{13}}{2}-4\right)\right)^{3} \nonumber \\
   &=
   30^{3}(31-9\sqrt{13})^{3}, \nonumber \\
j_{15}
   &=
   -\left(\frac{1+\sqrt{5}}{2}\right)^{2}
   \left(3\sqrt{5}\left(\frac{1+\sqrt{5}}{2}\right)\left(3\frac{1+\sqrt{5}}{2}-1\right)\right)^{3} \nonumber \\
   &=
   -3^{3}\left(\frac{1+\sqrt{5}}{2}\right)^{2}\left(5+4\sqrt{5}\right)^{3}, \nonumber \\
j_{25}
   &=
   \left(2^{2}\cdot 3\right)^{3} \nonumber \\
   & \quad \cdot 
   \left(\left(\frac{1-\sqrt{5}}{2} \right)^{4}\left(3\frac{1-\sqrt{5}}{2}-1\right)\left(7\frac{1-\sqrt{5}}{2}-2\right)\left(8\frac{1-\sqrt{5}}{2}-1\right)\right)^{3} \nonumber \\
   &=
   6^{3}(2927-1323\sqrt{5})^{3}, \nonumber \\
j_{35}
   &=
   -\left(2^{5}\sqrt{5}\left(\frac{1+\sqrt{5}}{2}\right)^{4}\right)^{3} 
   =
   -16^{3}(15+7\sqrt{5})^{3}, \nonumber \\
j_{37}
   &=
   \left(2^{2}\cdot 3\cdot 5(6-\sqrt{37})^{2}\right)^{3} \nonumber \\
   & \quad \cdot 
   \left(
   \left(3\frac{1-\sqrt{37}}{2}-10 \right)
   \left(9\frac{1-\sqrt{37}}{2}-31 \right)
   \left(12\frac{1-\sqrt{37}}{2}+29 \right)\right)^{3} \nonumber \\
   &=
   60^{3}(2837-468\sqrt{37})^{3}, \nonumber \\
j_{51}
   &=
   -(4+\sqrt{17})^{2}\left(2^{5}\cdot 3\cdot \left(4+\sqrt{17} \right)\left(\frac{-3+\sqrt{17}}{2}\right)\right)^{3} \nonumber \\
   &=
   -48^{3}(4+\sqrt{17})^{2}(5+\sqrt{17})^{3}, \nonumber \\
j_{75}
   &=
   -\sqrt{5}\left(2^{5}\cdot 3\left(\frac{1+\sqrt{5}}{2} \right)^{6}\left(3\frac{1+\sqrt{5}}{2}-1\right)\right)^{3} \nonumber \\
   &=
   -48^{3}\sqrt{5}(69+31\sqrt{5})^{3}, \nonumber \\
j_{91}
   &=
   -\left(2^{5}\cdot 3\left(\frac{3+\sqrt{13}}{2} \right)^{4}\left(3\frac{1+\sqrt{13}}{2}-5\right)\right)^{3} \nonumber \\
   &=
   -48^{3}(227+63\sqrt{13})^{3}, \nonumber \\
j_{99}
   &=
   -16^{3}(23+4\sqrt{33})^{2}\left(2^{5}\left(\frac{-5+\sqrt{33}}{2} \right)(11+2\sqrt{33})(4+\sqrt{33})\right)^{3} \nonumber \\
   &=
   -16^{3}(23+4\sqrt{33})^{2}(77+15\sqrt{33})^{3}, \nonumber \\
j_{115}
   &=
   \left(2^{5}\cdot 3\sqrt{5}\left(\frac{1+\sqrt{5}}{2}\right)^{10}\left(3\frac{1+\sqrt{5}}{2}-2\right)\right)^{3} \nonumber \\
   &=
   -48^{3}(785+351\sqrt{5})^{3}, \nonumber \\
j_{123}
   &=
   -\left(32+5\sqrt{41}\right)^{2}\left(2^{5}\cdot 3 \cdot 5\left(32+5\sqrt{41}\right)(-51+8\sqrt{41})\right)^{3} \nonumber \\
   &=
   -480^{3}(32+5\sqrt{41})^{2}(8+\sqrt{41})^{3}, \nonumber \\
j_{147}
   &=
   -3\sqrt{21}\left(2^{5}\cdot 3\cdot 5\left(\frac{5+\sqrt{21}}{2} \right)^{3}\left(\sqrt{21}-2\right)\right)^{3} \nonumber \\
   &=
   -3\cdot 480^{3}\sqrt{21}(142+31\sqrt{21})^{3}, \nonumber \\
j_{187}
   &=
   -\left(2^{5}\cdot 3\cdot 5 \sqrt{17}\left(4+\sqrt{17} \right)^{2}\left(\frac{3+\sqrt{17}}{2}\right)^{2}\right)^{3} \nonumber \\
   &=
   -240^{3}(3451+837\sqrt{17})^{3}, \nonumber \\
j_{235}
   &=
   -\left(2^{5}\cdot 3\sqrt{5}\cdot 11 \left(\frac{1+\sqrt{5}}{2} \right)^{14}\left(6\frac{1+\sqrt{5}}{2}-5\right)\right)^{3} \nonumber \\
   &=
   -528^{3}(8875+3969\sqrt{5})^{3}, \nonumber \\
j_{267}
   &=
   (500-53\sqrt{89})\left(2^{5}\cdot 3 \cdot 5\left(\frac{9+\sqrt{89}}{2} \right)^{2}\right)^{3} \nonumber \\
   & \quad \cdot \left((283+30\sqrt{89})(28+3\sqrt{89})(-113+12\sqrt{89})\right)^{3} \nonumber \\
   &=
   -240^{3}(500+53\sqrt{89})^{2}(625+53\sqrt{89})^{3}, \nonumber \\
j_{403}
   &=
   -\left(2^{5}\cdot 3\cdot 5 \right)^{3} \nonumber \\
   & \quad \cdot 
   \left(\left(\frac{3+\sqrt{13}}{2} \right)^{8}\left(3\frac{1+\sqrt{13}}{2}+2\right)\left(3\frac{1+\sqrt{13}}{2}+1\right)\left(6\frac{1+\sqrt{13}}{2}-11\right)\right)^{3} \nonumber \\
   &=
   -240^{3}(2809615+779247\sqrt{13})^{3}, \nonumber \\
j_{427}
   &=
   -\left(2^{5}\cdot 3\cdot 5\cdot 11 \left(\frac{39+5\sqrt{61}}{2} \right)^{2}\left(18\frac{1+\sqrt{61}}{2}+61\right)\left(3\frac{1+\sqrt{61}}{2}-11\right)\right)^{3} \nonumber \\
\label{eq:Berwick}
   &=
   -5280^{3}(236674+30303\sqrt{61})^{3}. 
\end{align}

In this paper, we propose the following three kinds of explicit formulas for the elliptic lambda function 
$$
\widetilde{\lambda }_{d}:=\lambda_{2} \left(\frac{1+\sqrt{-d}}{2}\right)=\lambda \left(\frac{\sqrt{-d}-1}{\sqrt{-d}+1}\right)
$$
by the elliptic modular functions $j_{d}$. 
For convenience, put
\begin{align}
a_{d}
   &:=
   \frac{1}{48}
   \left(\sqrt{1728-j_{d}}
   +\left(\beta _{d}-24\sqrt{3}j_{d}\right)^{\frac{1}{3}}+\left(\beta _{d}+24\sqrt{3}j_{d}\right)^{\frac{1}{3}}\right), \nonumber \\
b_{d}
   &:=
   \frac{1}{48}
   \sqrt{\left(\sqrt{3}\left(\left(\beta_{d}-24\sqrt{3}j_{d}\right)^{\frac{1}{3}}-\left(\beta_{d}+24\sqrt{3}j_{d}\right)^{\frac{1}{3}}\right)+24\right)^{2}+1152-9j_{d}}, \nonumber \\
c _{d}
   &:=
   \frac{1}{48}\sqrt{-2304t_{d}+1728-3j_{d}}. \nonumber 
\end{align} 
where
\begin{align}
\beta_{d}
   &:=
   \sqrt{1728j_{d}^{2}-j_{d}^{3}}, \nonumber \\
t_{d}
   &:=
   -\frac{1}{768}
   \left(\left(-884736j_{d}+2304j_{d}^{2}-j_{d}^{3}-12288\sqrt{3}\beta_{d}\right)^{\frac{1}{3}} \right. \nonumber \\
   & \quad \quad \quad \quad \quad \left. +\left(-884736j_{d}+2304j_{d}^{2}-j_{d}^{3}+12288\sqrt{3}\beta_{d}\right)^{\frac{1}{3}}\right). \nonumber
\end{align}
\begin{thm}
\label{thm:main theorem}
If $d\geq 3$, then all $a_{d}, b_{d}$ and $c_{d}$ are real number and 
\begin{align}
\label{eq:main theorem1}
\lambda \left(\frac{1+\sqrt{-d}}{2}\right)
   &=
   \frac{1}{\frac{1}{2}-\sqrt{-1}a_{d}}
   =
   \frac{1}{\frac{1}{2}-\sqrt{-1}b_{d}}
   =
   \frac{1}{\frac{1}{2}-\sqrt{-1}c_{d}}, \\
\label{eq:main theorem2}
\lambda \left(\frac{\sqrt{-d}-1}{\sqrt{-d}+1}\right)
   &=
   \frac{1}{2}+\sqrt{-1}a_{d}
   =
   \frac{1}{2}+\sqrt{-1}b_{d}
   =
   \frac{1}{2}+\sqrt{-1}c_{d}, \\
\label{eq:main theorem3}
\lambda \left(\frac{2}{1-\sqrt{-d}}\right)
   &=
   \frac{\sqrt{-1}a_{d}-\frac{1}{2}}{\sqrt{-1}a_{d}+\frac{1}{2}}
   =
   \frac{\sqrt{-1}b_{d}-\frac{1}{2}}{\sqrt{-1}b_{d}+\frac{1}{2}}
   =
   \frac{\sqrt{-1}c_{d}-\frac{1}{2}}{\sqrt{-1}c_{d}+\frac{1}{2}}, \\
\label{eq:main theorem4}
\lambda \left(-\frac{2}{1+\sqrt{-d}}\right)
   &=
   \frac{\sqrt{-1}a_{d}+\frac{1}{2}}{\sqrt{-1}a_{d}-\frac{1}{2}}
   =
   \frac{\sqrt{-1}b_{d}+\frac{1}{2}}{\sqrt{-1}b_{d}-\frac{1}{2}}
   =
   \frac{\sqrt{-1}c_{d}+\frac{1}{2}}{\sqrt{-1}c_{d}-\frac{1}{2}}, \\
\label{eq:main theorem5}
\lambda \left(-\frac{\sqrt{-d}+1}{\sqrt{-d}-1}\right)
   &=
   \frac{1}{2}-\sqrt{-1}a_{d}
   =
   \frac{1}{2}-\sqrt{-1}b_{d}
   =
   \frac{1}{2}-\sqrt{-1}c_{d}, \\
\label{eq:main theorem6} 
\lambda \left(\frac{-1+\sqrt{-d}}{2}\right)
   &=
   \frac{1}{\frac{1}{2}+\sqrt{-1}a_{d}}
   =
   \frac{1}{\frac{1}{2}+\sqrt{-1}b_{d}}
   =
   \frac{1}{\frac{1}{2}+\sqrt{-1}c_{d}}.
\end{align}
\end{thm}
These three kinds of expressions of (\ref{eq:main theorem1}) - (\ref{eq:main theorem6}) can be obtained by solving the sextic equation
\begin{align}
F(\lambda ,j)
   &:=
   256{\lambda}^{6}-768{\lambda}^{5}+(-j+1536 ){
\lambda}^{4}+ ( 2j-1792) {\lambda}^{3} \nonumber \\
\label{eq:sixth}
   & \quad 
   + ( -j+1536){\lambda}^{2}-768\lambda
   +256=0
\end{align}
in three different ways which are the methods of {\bf{(a)}} simplest cubic polynomials, {\bf{(b)}} Weber modular functions and {\bf{(c)}} Tschirnhaus transform. 
Although there are many investigations on singular values for the elliptic lambda functions, it seems that especially the formula by $a_{d}$ have been unknown.

As an application of Theorem \ref{thm:main theorem}, we derive the following cubic identities which we call ``{\it{incredible cubic identities}}''. 
In fact, by comparing the imaginary parts of (\ref{eq:main theorem2}), we obtain the incredible cubic identities. 
\begin{cor}[incredible cubic identities]
\label{thm:incredible cubic identities}
If $d\geq 3$, then we have
\begin{align}
\label{eq:incredible cubic identities}
\alpha _{d}
   :=
   \frac{1}{4}\left(\sqrt{\frac{1-\lambda (\sqrt{-d})}{\lambda (\sqrt{-d})}}-\sqrt{\frac{\lambda (\sqrt{-d})}{1-\lambda (\sqrt{-d})}}\right)
   =
   a_{d}=b_{d}=c_{d} \in \mathbb{R}_{\geq 0}. 
\end{align}
\end{cor}
Here the second equality of (\ref{eq:incredible cubic identities}) follows from (\ref{eq:Ochiai's formula}) of Lemma \ref{thm:key lem of lambda}. 
From this formula (\ref{eq:incredible cubic identities}), we obtain some curious exercises in high school mathematics. 
For example, by setting $d=11$, we obtain
\begin{align}
& 7\sqrt{11}+8((7\sqrt{11}+3\sqrt{3})^{\frac{1}{3}}+(7\sqrt{11}-3\sqrt{3})^{\frac{1}{3}}) \nonumber \\
   & \quad =
   \sqrt{(8\sqrt{3}((7\sqrt{11}+3\sqrt{3})^{\frac{1}{3}}-(7\sqrt{11}-3\sqrt{3})^{\frac{1}{3}})+3)^{2}+4626} \nonumber \\
   & \quad =
   \sqrt{48((35099+21\sqrt{33})^{\frac{1}{3}}+(35099-21\sqrt{33})^{\frac{1}{3}})+1563}. \nonumber \\
   & \quad =
   68.601585457080984363818472671223625016723649408286\cdots. \nonumber 
\end{align}

As a more important application of Theorem \ref{thm:main theorem}, we evaluate some singular values of $\widetilde{\lambda }_{d}$ corresponding to the results of Weber (\ref{eq:Weber}) and Berwick (\ref{eq:Berwick}) for the elliptic modular function. 
See Theorem \ref{thm:lambda Weber} and Theorem \ref{thm:lambda Werbick}.

The contents of this article are as follows. 
In Section 2, we give some formulas of the elliptic lambda function and the Cardano's formula for the proof of Theorem \ref{thm:main theorem}. 
In Section 3, we solve the sextic equation $F(\lambda ,j)=0$ in three different ways and prove Theorem \ref{thm:main theorem}. We also mention Ochiai's generalization \cite{O} of the incredible cubic identity $a_{d}=c_{d}$. 
As applications of Theorem \ref{thm:main theorem}, we give some singular values of the elliptic lambda function explicitly in Section 4. 


\section{Preliminaries}
First, we list the necessary kell-known formulas for the elliptic lambda function.
\begin{lem}
{\rm{(1)}} Landen transform \cite{V} p127, (7.1)
\begin{align}
\label{eq:Landen}
k\left(\frac{\tau }{2}\right)^{2}
   =
   \frac{4k(\tau )}{(1+k(\tau ))^{2}},
\end{align}
where $k(\tau )$ is the elliptic modulus defined by 
$$
k(\tau )
   :=
   4q^{\frac{1}{4}}
      \prod_{n=1}^{\infty}\frac{(1+q^{n})^{4}}{(1+q^{n-\frac{1}{2}})^{4}}.
$$
{\rm{(2)}} Singular value at $\tau =\sqrt{-1}$ 
\begin{align}
\label{eq:special values}
\lambda (\sqrt{-1})
   =
   \lambda \left(-\frac{1}{\sqrt{-1}}\right)
   =
   1-\lambda (\sqrt{-1})
   =
   \frac{1}{2}.
\end{align}
{\rm{(3)}} Derivation of $\lambda (\tau )$ \cite{V} p78, (4.78)
\begin{align}
\label{eq:derivation of k}
\frac{\lambda ^{\prime}(\tau )}{\lambda (\tau )}
   =
   \pi \sqrt{-1}q^{\frac{1}{2}}
   \prod_{n=1}^{\infty}
      (1-q^{n})^{4}(1-q^{n-\frac{1}{2}})^{8}. 
\end{align}
\end{lem}
Next, we prove key Lemma of the elliptic lambda function. 
\begin{lem}
\label{thm:key lem of lambda}
{\rm{(1)}} If $d>c>0$, then 
\begin{align}
\label{eq:key inequality}
0
   <
   \lambda (\sqrt{-d} )
   <
   \lambda (\sqrt{-c} )
   <1.
\end{align}
{\rm{(2)}} 
Put 
$$
\alpha _{d}
   :=
   \frac{1}{4}\left(\sqrt{\frac{1-\lambda (\sqrt{-d})}{\lambda (\sqrt{-d})}}-\sqrt{\frac{\lambda (\sqrt{-d})}{1-\lambda (\sqrt{-d})}}\right). 
$$
For any positive real number $d$, $\alpha _{d}$ is real and increases monotonically with $d$, and we have
\begin{align}
\label{eq:Ochiai's formula}
   \lambda \left(\frac{\sqrt{-d}-1}{\sqrt{-d}+1}\right)
   =
   \frac{1}{2}
   +\sqrt{-1}\alpha _{d}. 
\end{align}
\end{lem}
\begin{proof}
{\rm{(1)}} By the inversion formula 
$$
\lambda_{4}\left(\tau \right)
   :=
   \lambda \left(-\frac{1}{\tau }\right)
   =
   1-\lambda (\tau),
$$
it is enough to show that (\ref{eq:key inequality}) when $d>c>1$. 
First, for $x\geq 1$ we prove an inequality 
\begin{align}
\label{eq:assertion}
0<\lambda (\sqrt{-x})<1.
\end{align}
By the definition, 
$$
\lambda (\sqrt{-x})
   =
   16e^{-\pi \frac{\sqrt{x}}{2}}
\prod_{n=1}^{\infty}\left(\frac{1+e^{-\pi \sqrt{x}n}}{1+e^{-\pi \sqrt{x}\left(n-\frac{1}{2}\right)}}\right)^{8}
   >0.
$$
On the other hand, for any positive integer $n$ and $x\geq 1$, the inequality 
\begin{align}
\lambda (\sqrt{-x})
   <
   16e^{-\pi \frac{\sqrt{x}}{2}}
\prod_{n=1}^{2}\left(\frac{1+e^{-\pi \sqrt{x}n}}{1+e^{-\pi \sqrt{x}\left(n-\frac{1}{2}\right)}}\right)^{8}
   =:g(x)
    \nonumber
\end{align}
holds because 
$$
0<\frac{1+e^{-\pi \sqrt{x}n}}{1+e^{-\pi \sqrt{x}\left(n-\frac{1}{2}\right)}}< 1.
$$
Since $g(x)$ monotonically decreases and $g(1)=0.9730608\cdots$, $\lambda (\sqrt{-x})<1$ holds for $x>1$. 
Hence we have the inequality (\ref{eq:assertion}).

Also from this inequality (\ref{eq:assertion}) and the derivation of the elliptic lambda function (\ref{eq:derivation of k})
$$
\frac{d}{dx}\lambda (\sqrt{-x})
   =
   -\frac{\pi }{2\sqrt{x}}e^{-\pi \sqrt{x}}
   \lambda (\sqrt{-x})
   \prod_{n=1}^{\infty}
      (1-e^{-\pi n \sqrt{x}})^{4}(1-e^{-\pi (n-\frac{1}{2})\sqrt{x}})^{8},
$$
for any $x\geq 1$ we have 
$$
\frac{d}{dx}\lambda (\sqrt{-x})<0. 
$$
Then we obtain the conclusion (\ref{eq:key inequality}). \\
\noindent
{\rm{(2)}} From Landen transform (\ref{eq:Landen}) and modular transform 
$$
k(\tau -1)^{2}=\lambda (\tau -1)=\frac{\lambda (\tau)}{\lambda (\tau)-1}=\frac{k(\tau)^{2}}{k(\tau)^{2}-1},
$$
we have  
\begin{align}
\lambda \left(\frac{1+\sqrt{-d} }{2}\right)
   &=
   \frac{4k(1+\sqrt{-d} )}{(1+k(1+\sqrt{-d}))^{2}} \nonumber \\
   &=
   4\sqrt{-1}\sqrt{\frac{\lambda (\sqrt{-d})}{1-\lambda (\sqrt{-d})}}
   \frac{1}{\left(1+\sqrt{-1}\sqrt{\frac{\lambda (\sqrt{-d})}{1-\lambda (\sqrt{-d})}}\right)^{2}}. \nonumber 
\end{align}
Therefore we derive
\begin{align}
\lambda \left(\frac{\sqrt{-d}-1}{\sqrt{-d}+1}\right)
   &=
   \lambda \left(\frac{\frac{1+\sqrt{-d}}{2}-1}{\frac{1+\sqrt{-d}}{2}}\right) \nonumber \\
   &=
   \frac{\lambda \left(\frac{1+\sqrt{-d} }{2}\right)-1}{\lambda \left(\frac{1+\sqrt{-d} }{2}\right)} \nonumber \\
   &=
   \frac{4\sqrt{-1}\sqrt{\frac{\lambda (\sqrt{-d})}{1-\lambda (\sqrt{-d})}}-\left(1+\sqrt{-1}\sqrt{\frac{\lambda (\sqrt{-d})}{1-\lambda (\sqrt{-d})}}\right)^{2}}{4\sqrt{-1}\sqrt{\frac{\lambda (\sqrt{-d})}{1-\lambda (\sqrt{-d})}}} \nonumber \\
   &=
   \frac{1}{2}
   +\sqrt{-1}\alpha _{d}. \nonumber 
\end{align}
\end{proof}
As a Corollary of Lemma \ref{thm:key lem of lambda}, we derive behavior of $j_{d}$ for $d>0$. 
\begin{cor}
\label{thm:j behavior}
For any $d>0$, we have 
\begin{align}
j_{d}
   =
   -2^{6}\frac{(4\alpha _{d}^{2}-3)^{3}}{(4\alpha _{d}^{2}+1)^{2}}.
\end{align}
In particular, $j_{d}$ decreases monotonically on $d\geq 1$ and $j_{d}$ is non positive on $d\geq 3$. 
\end{cor}
Finally, we recall Cardano's formula. 
\begin{lem}[Cardano's formula]
\label{thm:Cardano}
Let 
$$
p:=b-\frac{a^{2}}{3}, \quad 
q:=c-\frac{ab}{3}+\frac{2a^{3}}{27}, \quad 
D:=-4p^{3}-27q^{2}=-2^{2}\cdot 3^{3}\left(\left(\frac{q}{2}\right)^{2}+\left(\frac{p}{3}\right)^{3}\right). 
$$
The three roots of a cubic polynomial 
$$
x^{3}+ax^{2}+bx+c
$$
can be written by 
\begin{equation}
\label{eq:Cardano}
-\frac{a}{3}
   +\omega ^{j}\left(-\frac{q}{2}+\sqrt{-\frac{D}{2^{2}3^{3}}}\right)^{\frac{1}{3}}
   +\omega ^{-j}\left(-\frac{q}{2}-\sqrt{-\frac{D}{2^{2}3^{3}}}\right)^{\frac{1}{3}} \quad j=0,1,2,
\end{equation}
where $\omega :=e^{\frac{2\pi \sqrt{-1}}{3}}$ and we define the branch of the cubic root of $-\frac{q}{2}\pm \sqrt{-\frac{D}{2^{2}3^{3}}}$ by 
$$
\left(-\frac{q}{2}+\sqrt{-\frac{D}{2^{2}3^{3}}}\right)^{\frac{1}{3}}\left(-\frac{q}{2}-\sqrt{-\frac{D}{2^{2}3^{3}}}\right)^{\frac{1}{3}}
   =
   -\frac{p}{3}.
$$
\end{lem}
\section{Proof of Theorem \ref{thm:main theorem}}
In this section, we give three kinds of expressions of $\widetilde{\lambda }_{d}$. 
Under the following, we assume $d\geq 3$. 
We give three expressions of the roots of the sextic polynomial $F(\lambda ,j_{d})$. \\

{\bf{(a)}} By the definition of the elliptic modular function $j(\tau)$ and its $\Gamma $-invariance, 
$F(\lambda ,j(\tau ))$ is decomposed as the product
$$
F(\lambda ,j(\tau ))
   =
   2^{8}\prod_{n=1}^{6}(\lambda -\lambda _{n}(\tau )).
$$
Hence the Galois group $\mathrm{Gal}(\mathbb{Q}(\lambda (\tau ))/\mathbb{Q}(j (\tau )))$ is generated by
$$
\sigma (\alpha ):=\frac{\alpha -1}{\alpha }, \quad \tau (\alpha ):=1-\alpha , \quad \alpha \in \mathbb{Q}(\lambda (\tau ))
$$
and is isomorphic to the symmetric group $\mathfrak{S}_{3}$.

We recall the simplest cubic polynomial
\begin{equation}
\label{eq:simplest cubic}
f(\lambda ,r):=\lambda ^{3}-r\lambda ^{2}+(r-3)\lambda +1
\end{equation}
and for any root $\alpha $ of (\ref{eq:simplest cubic}) 
$$
\sigma (\alpha )=\frac{\alpha  -1}{\alpha }, \quad \sigma ^{2}(\alpha )=\frac{1}{1-\alpha }
$$
are also roots of $f(\lambda ,r)$. 
So if 
$$
r_{\pm }
   :=
   \frac{3}{2}\pm \frac{\sqrt{j(\tau )-1728}}{16}
   =
   \frac{3}{2}\pm \frac{\lambda (\tau )^{3}-\frac{3}{2}\lambda (\tau )^{2}-\frac{3}{2}\lambda (\tau )+1}{\lambda (\tau)(1-\lambda (\tau ))} \nonumber 
$$
which are the roots of the quadratic polynomial $2^{8}(r^{2}-3r+9)-j(\tau )$, then $F(\lambda ,j_{d})$ is decomposed into two products of simplest cubic polynomials
$$
F(\lambda ,j(\tau ))=2^{8}f(\lambda ,r_{\pm })f(\lambda ,3-r_{\pm }). 
$$
Therefore to solve the sextic polynomial $F(\lambda ,j(\tau ))$ it is enough to solve the simplest cubic polynomial $f(\lambda ,j(\tau ))$.

From Lemma \ref{thm:key lem of lambda}, if $\tau =\frac{1+\sqrt{-d}}{2}$ for $d>3$, then there is only one root of $F(\lambda ,j_{d})$ whose real part is $\frac{1}{2}$ and imaginary part is positive, and the root is equal to $\widetilde{\lambda }_{d}$. 
Then, by applying Lemma \ref{thm:Cardano} (Cardano's formula), we search for the root of such a property and find that 
$$
\frac{1}{2}+\sqrt{-1}a_{d}.
$$
is it. 
From Corollary \ref{thm:j behavior} and the definition of $a_{d}$, if $d\geq 3$, then $a_{d}$ is real and increases monotonically with $d$. \\

{\bf{(b)}} Next we give the second solution of the the sextic equation $F(\lambda ,j_{d})$. 
Since $F(\lambda ,j_{d})=0$ is a reciprocal equation, it be reduced to the cubic equation 
\begin{equation}
\label{eq:Weber cubic}
z^{3}-\frac{j_{d}}{256}z+\frac{j_{d}}{256}=0
\end{equation}
by putting $z=\lambda +\lambda ^{-1}-1$. 
If $d=3$, i.e. $j_{d}=0$, then the cubic equation (\ref{eq:Weber cubic}) has a triple root $z=0$. 
When $d>3$, i.e. $j_{d}<0$, this equation (\ref{eq:Weber cubic}) has only one real root $z_{0}$ which is expressed by
\begin{equation}
\label{eq:Weber expression}
z_{d}
   =
   \frac{1}{48}\left(\left(\beta_{d}-24\sqrt{3}j_{d}\right)^{\frac{1}{3}}-\left(\beta_{d}+24\sqrt{3}j_{d}\right)^{\frac{1}{3}}\right) \in \mathbb{R}
\end{equation}
and satisfies 
\begin{equation}
\label{eq:z1 ineq}
0\leq z_{d}<1.
\end{equation}
Here $z_{d}$ is equal to $0$ if and only if $d=3$.

By substituting $z_{d}$ into $z$ of the quadratic equation
$$
\lambda ^{2}-(z+1)\lambda +1=0,
$$
we have
$$
\frac{1+z_{d}\pm \sqrt{(3+z_{d})(z_{d}-1)}}{2}.
$$
However, the real part of these two roots are not equal to $\frac{1}{2}$. 
The root whose real part is $\frac{1}{2}$ and imaginary part is positive, is given by 
$$
\frac{1}{1-\frac{1+z_{d}+ \sqrt{(3+z_{d})(z_{d}-1)}}{2}}.
$$
In fact, a simple calculation shows that 
\begin{align}
\frac{1}{1-\frac{1+z_{d}+ \sqrt{(3+z_{d})(z_{d}-1)}}{2}}
   &=
   \frac{2}{1-z_{d}-\sqrt{-1}\sqrt{(3+z_{d})(1-z_{d})}} \nonumber \\
   &=
   \frac{1}{2}+\frac{\sqrt{-1}}{2}\frac{\sqrt{(3+z_{d})(1-z_{d})}}{1-z_{d}} \nonumber \\
   &=
   \frac{1}{2}+\frac{\sqrt{-1}}{2}\sqrt{(2z_{d}+1)^{2}+2-\frac{j_{d}}{64}} \nonumber \\
   &=
   \frac{1}{2}+\sqrt{-1}b_{d}. \nonumber 
\end{align}
Here the third equality follows from the inequality $(\ref{eq:z1 ineq})$ and 
\begin{align}
\sqrt{\frac{3+z_{d}}{1-z_{d}}}
   &=
   \sqrt{\frac{(3+z_{d})\left(z_{d}^{2}+z_{d}+1-\frac{j_{d}}{256}\right)}{(1-z_{d})\left(z_{d}^{2}+z_{d}+1-\frac{j_{d}}{256}\right)}} 
   =
   \sqrt{4z_{d}^{2}+4z_{d}+3-\frac{j_{d}}{64}}. \nonumber
\end{align}
\begin{rem}
We recall the Weber modular functions $\mathfrak{f}(\tau )$, $\mathfrak{f}_{1}(\tau )$ and $\mathfrak{f}_{2}(\tau )$ defined by 
\begin{align}
\mathfrak{f}(\tau )
   &:=
   q^{-\frac{1}{48}}\prod_{n=1}^{\infty}(1+q^{n-\frac{1}{2}})
   =
   e^{-\frac{\pi \sqrt{-1}}{24}}
   \frac{\eta \left(\frac{\tau +1}{2}\right)}{\eta(\tau )}
   =
   \frac{\eta \left(\tau \right)^{2}}{\eta \left(\frac{\tau }{2}\right)\eta(2\tau )}, \nonumber \\
\mathfrak{f}_{1}(\tau )
   &:=
   q^{-\frac{1}{48}}\prod_{n=1}^{\infty}(1-q^{n-\frac{1}{2}})
   =
   \frac{\eta \left(\frac{\tau }{2}\right)}{\eta(\tau )}, \nonumber \\
\mathfrak{f}_{2}(\tau )
   &:=
   \sqrt{2}q^{\frac{1}{48}}\prod_{n=1}^{\infty}(1+q^{n})
   =
   \sqrt{2}\frac{\eta \left(2\tau \right)}{\eta(\tau )}, \nonumber
\end{align}
where $\eta(\tau )$ is the Dedekind eta function
\begin{align}
\eta(\tau )
   &:=
   q^{\frac{1}{24}}\prod_{n=1}^{\infty}(1-q^{n}). \nonumber 
\end{align}
Weber modular function satisfy the following function equations
\begin{align}
\mathfrak{f}_{1}(\tau )^{8}+\mathfrak{f}_{2}(\tau )^{8}
   &=
   \mathfrak{f}(\tau )^{8}, \nonumber \\
\mathfrak{f}(\tau )\mathfrak{f}_{1}(\tau )\mathfrak{f}_{2}(\tau )
   &=
   \sqrt{2}. \nonumber 
\end{align}
From these well-known function equations and the relationship between $\lambda (\tau)$ and the Weber modular functions 
$$
\lambda (\tau)
   =
   \frac{\mathfrak{f}_{1}(\tau )^{8}}{\mathfrak{f}(\tau )^{8}}, 
$$
we have 
\begin{align}
\frac{x}{16}
   :=&
   z-1 \nonumber \\
   =&
   \lambda (\tau )+\lambda (\tau )^{-1}-2 \nonumber \\
   =&
   \frac{(\mathfrak{f}(\tau )^{8}-\mathfrak{f}_{1}(\tau )^{8})^{2}}{\mathfrak{f}(\tau )^{8}\mathfrak{f}_{1}(\tau )^{8}} \nonumber \\
   =&
   \frac{(\mathfrak{f}_{2}(\tau )^{8})^{2}}{\mathfrak{f}(\tau )^{8}\mathfrak{f}_{1}(\tau )^{8}} \nonumber \\
   =&
   \frac{\mathfrak{f}_{2}(\tau )^{24}}{\mathfrak{f}(\tau )^{8}\mathfrak{f}_{1}(\tau )^{8}\mathfrak{f}_{2}(\tau )^{8}} \nonumber \\
   =&
   \frac{\mathfrak{f}_{2}(\tau )^{24}}{16}. \nonumber
\end{align}
More generally, three roots of the cubic polynomial
$$
z^{3}-\frac{j(\tau )}{256}z+\frac{j(\tau )}{256}
   =
   \frac{1}{16^{3}}((x+16)^{3}-j(\tau )x)
$$
are $-\mathfrak{f}(\tau )^{24}$, $\mathfrak{f}_{1}(\tau )^{24}$ and $\mathfrak{f}_{2}(\tau )^{24}$ \cite{W}. 
Hence, the second solution {\rm{(b)}} is regarded as applying (the 24th power of) Weber modular functions. 
\end{rem}

{\bf{(c)}} Finally, we consider a Tschirnhaus transform
\begin{align}
\label{eq:Tschirnhaus quadratic}
t=\lambda ^{2}-\lambda +1-\frac{j_{d}}{768}.
\end{align}
By this transform, the sextic polynomial $F(\lambda ,j_{d})=0$ be reduced to the cubic polynomial
\begin{equation}
\label{eq:Tschirnhaus cubic}
256t^{3}+j_{d}\left(2-\frac{j_{d}}{768}\right)t-j_{d}\left(1-\frac{j_{d}}{384}+\frac{j_{d}^{2}}{884736}\right).
\end{equation}
Similar to the solution {\rm{(b)}}, when $d=3$, i.e. $j_{d}=0$, then the cubic polynomial (\ref{eq:Tschirnhaus cubic}) has a triple root $t=0$. 
If $d>3$, i.e. $j_{d}<0$, then (\ref{eq:Tschirnhaus cubic}) has only one real root which is negative. 
From the Cardano's formula (\ref{eq:Cardano}), this root is equal to $t_{d}$.

By substituting $t_{d}$ into $t$ of the quadratic equation (\ref{eq:Tschirnhaus quadratic}), we obtain the root whose real part is $\frac{1}{2}$ and imaginary part is positive
$$
\frac{1+\sqrt{4t_{d}+\frac{j_{d}}{192}-3}}{2}=\frac{1}{2}+\sqrt{-1}c_{d}. 
$$
This completes the proof of the theorem \ref{thm:main theorem}. 
The incredible cubic identities $a_{d}=b_{d}=c_{d}$ (\ref{eq:incredible cubic identities}) follow form comparing the imaginary parts of (\ref{eq:main theorem2}). 
Also see the following remark. 
\begin{rem}
Hiroyuki Ochiai \cite{O} pointed out that the incredible cubic identity $a_{d}=c_{d}$ can be generalized as follows. 
Let $r$, $x$ and $y$ be complex numbers with 
$$
((2r+y)^{3}x^{2}y)^{\frac{1}{3}}=(2r+y)(x^{2}y)^{\frac{1}{3}}
$$
and 
$$
((2r+x)^{3}xy^{2})^{\frac{1}{3}}=(2r+x)(xy^{2})^{\frac{1}{3}}. 
$$
Put
\begin{align}
a
   &:=
   r+(x^{2}y)^{\frac{1}{3}}+(xy^{2})^{\frac{1}{3}}, \nonumber \\
c
   &:=
   \sqrt{r^{2}+2xy+((2r+y)^{3}x^{2}y)^{\frac{1}{3}}+((2r+x)^{3}xy^{2})^{\frac{1}{3}}}. \nonumber
\end{align}
An easy calculation shows that $a=c$. 
In particular, by setting 
$$
r=\sqrt{1728-j_{d}}, \,\,
x=24\sqrt{3}-\frac{\sqrt{1728j_{d}^{2}-j_{d}^{3}}}{j_{d}}, \,\, y=-24\sqrt{3}-\frac{\sqrt{1728j_{d}^{2}-j_{d}^{3}}}{j_{d}},
$$
we obtain our incredible cubic identity $a_{d}=c_{d}$. 
\end{rem}

\section{Singular values of $\widetilde{\lambda }_{d}$}
Singular values of $\widetilde{\lambda }_{d}=\frac{1}{2}+a_{d}$ corresponding to the singular values $j_{d}$ in the list (\ref{eq:Weber}) are as follows. 
\begin{thm}
\label{thm:lambda Weber}
\begin{align}
\label{eq:l3}
\widetilde{\lambda }_{3}
   &=
   \frac{1}{2}
   +\frac{1}{2}\sqrt{-3}, \\
\label{eq:l7}
\widetilde{\lambda }_{7}
   &=
   \frac{1}{2}
   +\frac{3}{2}\sqrt{-7}, \\
\label{eq:l11}
\widetilde{\lambda }_{11}
   &=
   \frac{1}{2}
   +\frac{7}{6}\sqrt{-11}
   +\frac{4}{3}\sqrt{-1}\left((7\sqrt{11}+3\sqrt{3})^{\frac{1}{3}}+(7\sqrt{11}-3\sqrt{3})^{\frac{1}{3}}\right), \\
\label{eq:l19}
\widetilde{\lambda }_{19}
   &=
   \frac{1}{2}
   +\frac{9}{2}\sqrt{-19}
   +4\sqrt{-1}\left((27\sqrt{19}+3\sqrt{3})^{\frac{1}{3}}+(27\sqrt{19}-3\sqrt{3})^{\frac{1}{3}}\right), \\
\label{eq:l27}
\widetilde{\lambda }_{27}
   &=
   \frac{1}{2}
   +\frac{1}{6}\sqrt{-3}(253+200\sqrt[3]{2}+160\sqrt[3]{4}), \\
\widetilde{\lambda }_{43}
   &=
   \frac{1}{2}
   +\frac{189}{2}\sqrt{-43} \nonumber \\
\label{eq:l43}
   & \quad +
   +40\sqrt{-1}\left((567\sqrt{43}+3\sqrt{3})^{\frac{1}{3}}+(567\sqrt{43}-3\sqrt{3})^{\frac{1}{3}}\right), \\
\widetilde{\lambda }_{67}
   &=
   \frac{1}{2}
   +\frac{1953}{2}\sqrt{-67} \nonumber \\
\label{eq:l67}
   & \quad +
   220\sqrt{-1}\left((5859\sqrt{67}+3\sqrt{3})^{\frac{1}{3}}+(5859\sqrt{67}-3\sqrt{3})^{\frac{1}{3}}\right), \\
\widetilde{\lambda }_{163}
   &=
   \frac{1}{2}
   +\frac{1672209}{2}\sqrt{-163} \nonumber \\
\label{eq:l163}
   & \quad +
   26680\sqrt{-1}\left((5016627\sqrt{163}+3\sqrt{3})^{\frac{1}{3}}+(5016627\sqrt{163}-3\sqrt{3})^{\frac{1}{3}}\right).
\end{align}
\end{thm}
\begin{rem}
For $d=7$ i.e. $j_{d}=-15^{3}$, from Theorem \ref{thm:main theorem} we obtain the expression of $\lambda _{7}$ as
\begin{equation}
\label{eq:l7 1}
\widetilde{\lambda }_{7}
   =
   \frac{1}{2}
   +\frac{9}{16}\sqrt{-7}
   +\sqrt{-1}\frac{5}{16}\left( (27\sqrt{7}+24\sqrt{3})^{\frac{1}{3}}+(27\sqrt{7}-24\sqrt{3})^{\frac{1}{3}}\right).
\end{equation}
On the other hand, the sextic polynomial $F(\lambda ,j_{7})$ is reducible on $\mathbb{Q}$ and is factored as 
\begin{align}
F(\lambda ,j_{7})
   &=
   256{\lambda}^{6}-768{\lambda}^{5}+4911{
\lambda}^{4}-8542{\lambda}^{3}+4911{\lambda}^{2}-768\lambda
   +256 \nonumber \\
   &=
   (16\lambda ^{2}-31\lambda +16)
   (16\lambda ^{2}-\lambda +1)
   (\lambda ^{2}-\lambda +16). \nonumber
\end{align}
From this decomposition, we have (\ref{eq:l7}). 
By comparing (\ref{eq:l7 1}) and (\ref{eq:l7}), we derive another kind of (incredible) cubic relation
$$
(27\sqrt{7}+24\sqrt{3})^{\frac{1}{3}}+(27\sqrt{7}-24\sqrt{3})^{\frac{1}{3}}=3\sqrt{7}
$$
or
$$
(3\sqrt{21}+8)^{\frac{1}{3}}+(3\sqrt{21}-8)^{\frac{1}{3}}=\sqrt{21}.
$$
\end{rem}
For convenience, we define the sets $D_{1}$ and $D_{2}$ by 
\begin{align}
D_{1}
   &:=
   \{35,51,75,91,99,115,123,147,187,235,267,403,427\}, \nonumber \\
D_{2}
   &:=
   \{4,5,9,13,15,25,37\} \nonumber 
\end{align}
and put
\begin{align}
x_{35,\pm}
   &:=(115\sqrt{35}+256\sqrt{7}\pm 3\sqrt{3})^{\frac{1}{3}}, \nonumber \\
x_{91,\pm}
   &:=(21087\sqrt{91}+76032\sqrt{7}\pm 3\sqrt{3})^{\frac{1}{3}}, \nonumber \\
x_{115,\pm}
   &:=(120555\sqrt{115}+269568\sqrt{23}\pm 3\sqrt{3})^{\frac{1}{3}}, \nonumber \\
x_{187,\pm}
   &:=(9744111\sqrt{187}+40176000\sqrt{11}\pm 3\sqrt{3})^{\frac{1}{3}}, \nonumber \\
x_{235,\pm}
   &:=(116974935\sqrt{235}+261563904\sqrt{47}\pm 3\sqrt{3})^{\frac{1}{3}}, \nonumber \\
x_{403,\pm}
   &:=(154191193947\sqrt{403}+555944256000\sqrt{31}\pm 3\sqrt{3})^{\frac{1}{3}}, \nonumber \\
x_{427,\pm}
   &:=(377909472375\sqrt{427}+2951567334144\sqrt{7}\pm 3\sqrt{3})^{\frac{1}{3}} \nonumber  \\
x_{51,\pm}
   &:=(217\sqrt{17}+897)^{\frac{1}{3}+\frac{1\pm 1}{6}}(217\sqrt{17}+895)^{\frac{1}{3}+\frac{1\mp 1}{6}}, \nonumber \\
x_{123,\pm}
   &:=(69125\sqrt{41}+442623)^{\frac{1}{3}+\frac{1\pm 1}{6}}(69125\sqrt{41}+442625)^{\frac{1}{3}+\frac{1\mp 1}{6}}, \nonumber \\
x_{267,\pm}
   &:=(178875053\sqrt{89}+1687504001)^{\frac{1}{3}+\frac{1\pm 1}{6}} \nonumber \\
   & \quad \cdot (178875053\sqrt{89}+1687503999)^{\frac{1}{3}+\frac{1\mp 1}{6}}. \nonumber
\end{align}
Singular values of $\widetilde{\lambda }_{d}=\frac{1}{2}+a_{d}$ corresponding to the singular values $j_{d}$ in the list (\ref{eq:Berwick}) are as follows. 
\begin{thm}
\label{thm:lambda Werbick}
If $d \in D_{1}$ is square free and odd not divisible by 3, then we have
\begin{align}
\widetilde{\lambda }_{35}
   &=
   \frac{1}{2}
   +\sqrt{-1}\left(\frac{128}{3}\sqrt{7}+\frac{115}{6}\sqrt{35}\right) \nonumber \\
\label{eq:l35}
   & \quad 
   +\sqrt{-1}\left(10+\frac{14}{3}\sqrt{5} \right)\left(x_{35,+}+x_{35,-}\right), \\
\widetilde{\lambda }_{91}
   &=
   \frac{1}{2}
   +\sqrt{-1}\left(12672\sqrt{7}+\frac{7029}{2}\sqrt{91}\right) \nonumber \\
\label{eq:l91}
   & \quad 
   +\sqrt{-1}\left(454+126\sqrt{13} \right)\left(x_{91,+}+x_{91,-}\right), \\   
\widetilde{\lambda }_{115}
   &=
   \frac{1}{2}
   +\sqrt{-1}\left(44928\sqrt{23}+\frac{40185}{2}\sqrt{115}\right) \nonumber \\
\label{eq:l115}
   & \quad 
   +\sqrt{-1}\left(1570+702\sqrt{5} \right)\left(x_{115,+}+x_{115,-}\right), \\
\widetilde{\lambda }_{187}
   &=
   \frac{1}{2}
   +\sqrt{-1}\left(6696000\sqrt{11}+\frac{3248037}{2}\sqrt{187}\right) \nonumber \\
\label{eq:l187}
   & \quad 
   +\sqrt{-1}\left(34510+8370\sqrt{17} \right)\left(x_{187,+}+x_{187,-}\right), \\
\widetilde{\lambda }_{235}
   &=
   \frac{1}{2}
   +\sqrt{-1}\left(43593984\sqrt{47}+\frac{38991645}{2}\sqrt{235}\right) \nonumber \\
\label{eq:l235}
   & \quad 
   +\sqrt{-1}\left(195250+87318\sqrt{5} \right)\left(x_{235,+}+x_{235,-}\right), \\
\widetilde{\lambda }_{403}
   &=
   \frac{1}{2}
   +\sqrt{-1}\left(92657376000\sqrt{31}+\frac{51397064649}{2}\sqrt{403}\right) \nonumber \\
\label{eq:l403}
   & \quad 
   +\sqrt{-1}\left(28096150+7792470\sqrt{13} \right)\left(x_{403,+}+x_{403,-}\right), \\
\widetilde{\lambda }_{427}
   &=
   \frac{1}{2}
   +\sqrt{-1}\left(491927889024\sqrt{7}+\frac{125969824125}{2}\sqrt{427}\right) \nonumber \\
\label{eq:l427}
   & \quad 
   +\sqrt{-1}\left(52068280+6666660\sqrt{61} \right)\left(x_{427,+}+x_{427,-}\right).
\end{align}
If $d \in D_{1}$ is square free and odd divisible by 3, then we have
\begin{align}
\widetilde{\lambda }_{51}
   &=
   \frac{1}{2}
   +\sqrt{-1}\left(448\sqrt{3}+\frac{217}{2}\sqrt{51}\right) \nonumber \\
\label{eq:l51}
   & \quad
   +\sqrt{-1}\frac{\sqrt{3}}{2}\left(x_{51,+}+x_{51,-}\right), \\
\widetilde{\lambda }_{123}
   &=
   \frac{1}{2}
   +\sqrt{-1}\left(221312\sqrt{3}+\frac{69125}{2}\sqrt{123}\right) \nonumber \\
\label{eq:l123}
   & \quad 
   +\sqrt{-1}\frac{\sqrt{3}}{2}\left(x_{123,+}+x_{123,-}\right), \\
\widetilde{\lambda }_{267}
   &=
   \frac{1}{2}
   +\sqrt{-1}\left(843752000\sqrt{3}+\frac{178875053}{2}\sqrt{267}\right) \nonumber \\
\label{eq:l267}
   & \quad 
   +\sqrt{-1}\frac{\sqrt{3}}{2}\left(x_{427,+}+x_{427,-}\right).
\end{align}
If $d \in D_{1}$ is not square free, then 
\begin{align}
\widetilde{\lambda }_{75}
   &=
   \frac{1}{2}
   +\sqrt{-1}\sqrt{3}\left(\frac{9729}{2}+2176\sqrt{5}\right) \nonumber \\
   & \quad
      +\sqrt{-1}\sqrt{3}(69\cdot 5^{\frac{1}{6}}+31\cdot 5^{\frac{2}{3}}) \nonumber \\
\label{eq:l75}
      & \quad \quad \cdot 
      \left(2^{\frac{4}{3}}\left(4865+2176\sqrt{5}\right)^{\frac{1}{3}}
         +2^{\frac{11}{3}}\left(38+17\sqrt{5}\right)^{\frac{1}{3}}\right), \\
\widetilde{\lambda }_{99}
   &=
   \frac{1}{2}
   +\sqrt{-1}\left(\frac{110656}{3}\sqrt{3}+\frac{115577}{6}\sqrt{11}\right) \nonumber \\
   & \quad
   +\sqrt{-1}\left(\frac{7502}{3}+\frac{1306}{3}\sqrt{33}\right)\left(4719\sqrt{3}+2563\sqrt{11}\right)^{\frac{1}{3}} \nonumber \\
\label{eq:l99}
   & \quad
   +\sqrt{-1}\left(\frac{5425}{96}\sqrt{3}+\frac{31169}{1056}\sqrt{11}\right)\left(4719\sqrt{3}+2563\sqrt{11}\right)^{\frac{2}{3}}, \\
\widetilde{\lambda }_{147}
   &=
   \frac{1}{2}
   +\sqrt{-1}\left(\frac{2245375}{2}\sqrt{3}+734976\sqrt{7}\right) \nonumber \\
   & \quad 
   +\sqrt{-1}(11360\sqrt{3}\cdot 7^{\frac{1}{6}}+7440\cdot 7^{\frac{2}{3}})\left(105252\sqrt{3}+68904\sqrt{7}\right)^{\frac{1}{3}} \nonumber \\
\label{eq:l147}
   & \quad 
   +\sqrt{-1}(8520\sqrt{3}\cdot 7^{\frac{1}{6}}+5580\cdot 7^{\frac{2}{3}})\left(249486 \sqrt{3}+163328\sqrt{7}\right)^{\frac{1}{3}}.
\end{align}
For $d \in D_{2}$, we have
\begin{align}
\label{eq:l4}
\widetilde{\lambda }_{4}
   &=
   \frac{1}{2}
   +\sqrt{-1}\cdot \frac{3}{16}\sqrt{24+22\sqrt{2}}, \\
\label{eq:l5}
\widetilde{\lambda }_{5}
   &=
   \frac{1}{2}
   +\sqrt{-1}\sqrt{2+\sqrt{5}}, \\
\label{eq:l9}
\widetilde{\lambda }_{9}
   &=
   \frac{1}{2}
   +\sqrt{-1}\sqrt{24+14\sqrt{3}}, \\
\label{eq:l13}
\widetilde{\lambda }_{13}
   &=
   \frac{1}{2}
   +\sqrt{-1}\cdot 3\sqrt{18+5\sqrt{13}}, \\
\label{eq:l15}
\widetilde{\lambda }_{15}
   &=
   \frac{1}{2}
   +8\sqrt{-3}+\frac{7}{2}\sqrt{-15}, \\
\label{eq:l25}
\widetilde{\lambda }_{25}
   &=
   \frac{1}{2}
   +\sqrt{-1}\cdot 6\sqrt{360+161\sqrt{5}}, \\
\label{eq:l37}
\widetilde{\lambda }_{37}
   &=
   \frac{1}{2}
   +\sqrt{-1}\cdot 21\sqrt{882+145\sqrt{37}}.
\end{align}
\end{thm}
\begin{rem}
We remark that for the case of $d \in D_{2}$, the expressions obtained from Theorem \ref{thm:main theorem} are complicated, and (\ref{eq:l4}) - (\ref{eq:l37}) can not be obtained directly. 
In this case, the sextic polynomial $F(\lambda ,j_{d})$ is decomposed into the product to three quadratic polynomials on $\mathbb{Q}(j_{d})$. 
In fact, it is easy to show that 
\begin{align}
F(\lambda ,j_{4})
   &=
   2(\lambda ^{2}-(6336\sqrt{2}-8960)(\lambda -1)) \nonumber \\
   & \quad \cdot (128\lambda ^{2}-128\lambda +140+99\sqrt{2})
   (\lambda ^{2}+(6336\sqrt{2}-8962)\lambda +1), \nonumber \\
F(\lambda ,j_{5})
   &=
   64(\lambda ^{2}+(16\sqrt{5}-36)(\lambda -1)) \nonumber \\
   & \quad \cdot (4\lambda ^{2}-4\lambda +9+4\sqrt{5})
   (\lambda ^{2}+(34-16\sqrt{5})\lambda +1), \nonumber \\
   F(\lambda ,j_{9})
   &=
   64(\lambda ^{2}+(224\sqrt{3}-388)(\lambda -1)) \nonumber \\
   & \quad \cdot (4\lambda ^{2}-4\lambda +97+56\sqrt{3})
   (\lambda ^{2}-(224\sqrt{3}-386)\lambda +1), \nonumber \\
F(\lambda ,j_{13})
   &=
   64(\lambda ^{2}+(720\sqrt{13}-2596)(\lambda -1)) \nonumber \\
   & \quad \cdot (4\lambda ^{2}-4\lambda +649+180\sqrt{13})
   (\lambda ^{2}+(2594-720\sqrt{13})\lambda +1), \nonumber \\
F\left(\lambda ,j_{15}\right)
   &=
   \frac{1}{4}(32\lambda ^{2}-(47-21\sqrt{5})(\lambda -1)) \nonumber \\
   & \quad \cdot (\lambda ^{2}-\lambda +376+168\sqrt{5})
   (32\lambda ^{2}-(17+21\sqrt{5})\lambda +32), \nonumber \\
F\left(\lambda ,j_{25}\right)
   &=
   64(\lambda ^{2}+(92736\sqrt{5}-207364)(\lambda -1)) \nonumber \\
   & \quad \cdot (4\lambda^{2}-4\lambda +23184\sqrt{5}+51841)
   (\lambda^{2}-(92736\sqrt{5}-207362)\lambda +1), \nonumber \\
F(\lambda ,j_{37})
   &=
   64(\lambda ^{2}+(1023120\sqrt{37}-6223396)(\lambda -1)) \nonumber \\
   & \quad \cdot 
   (4\lambda ^{2}-4\lambda +1555849+255780\sqrt{37}) \nonumber \\
   & \quad \cdot 
   (\lambda ^{2}+(6223394-1023120\sqrt{37})\lambda +1). \nonumber
\end{align}
From these decompositions, we obtain (\ref{eq:l4}) - (\ref{eq:l37}) without Theorem \ref{thm:main theorem}. 
\end{rem}


\noindent 
Department of Mathematics, Graduate School of Science, Kobe University, \\
1-1, Rokkodai, Nada-ku, Kobe, 657-8501, JAPAN\\
E-mail: g-shibukawa@math.kobe-u.ac.jp


\begin{thebibliography}{99}
\bibitem[1]{B1}
{B. C. Berndt} :
{\em Ramanujan's Notebooks Part V}, 
{Springer, 1998}. 
\bibitem[2]{B}
{W. Berwick} :
{\em Modular invariants expressible in terms of quadratic and cubic irrationalities}, 
{Proc. London Math. Soc. {\bf{2}}-1 (1928), 53--69}. 
\bibitem[3]{C}
{D. Cox} :
{\em A Primes of the form $x^{2}+ ny^{2}$ : Fermat, class field theory, and complex multiplication}, 
{John Wiley $\&$ Sons, 1989}.
\bibitem[4]{O}
{H. Ochiai} : 
{\em private note}. 
\bibitem[5]{R}
{R. Russell} :
{\em On modular equations}, 
{Proc. London Math. Soc. {\bf{21}}-1 (1890), 351--395}. 
\bibitem[6]{S}
{J. Silverman} :
{\em Advanced Topics in the Arithmetic of Elliptic Curves}, 
{Springer, GTM {\bf{151}}, 1994}.
\bibitem[7]{V}
{K. Venkatachaliengar} : 
{\em Development of elliptic functions according to Ramanujan}, 
{World Scientific, Monographs in Number Theory {\bf{6}}, 2011}. 
\bibitem[8]{W}
{H.\,Weber} :
{\em Lehrbuch der Algebra III}, 
{Vieweg Braunschweig, 1908}.
\end{thebibliography}
\end{document}